\documentclass[preprint,authoryear,12pt]{elsarticle}

\usepackage{verbatim} 
\usepackage{amssymb}
\usepackage{amsmath}
\usepackage{amsthm}
\usepackage{epstopdf}
\usepackage{setspace}
\usepackage{rotating}
\usepackage[T1]{fontenc}
\usepackage{subfig}
\usepackage{algorithm,algorithmic}
\usepackage{chngcntr} 
\usepackage{lineno}
\usepackage[framed,numbered,autolinebreaks,useliterate]{mcode}
\newtheorem{definition}{Definition}[section]

\newtheorem{theorem}{Theorem}[section]
\newtheorem{thm}{Theorem}[section]
\newtheorem{cor}[thm]{Corollary}

\numberwithin{equation}{section}
\usepackage{rotating}

\usepackage{epstopdf}
\usepackage{adjustbox}
\usepackage{blindtext}
\journal{arXiv.org}
\begin{document}
\begin{frontmatter}
\title{A Fast Algorithm for Computing the Fourier Spectrum of a Fractional Period}
\author{Jiasong Wang$^{1}$, Changchuan Yin$^{2,\ast}$}
\address{1.Department of Mathematics, Nanjing University, Nanjing, Jiangsu 210093, China\\
2.Department of Mathematics, Statistics and Computer Science\\ The University of Illinois at Chicago, Chicago, IL 60607-7045, USA\\
$\ast$ Corresponding author, Email: cyin1@uic.edu
}
\begin{abstract}
The Fourier spectrum at a fractional period is often examined when extracting features from biological sequences and time series. It reflects the inner information structure of the sequences. A fractional period is not uncommon in time series. A typical example is the 3.6 period in protein sequences, which determines the $\alpha$-helix secondary structure. Computing the spectrum of a fractional period offers a high-resolution insight into a time series. It has thus become an important approach in genomic analysis. However, computing Fourier spectra of fractional periods by the traditional Fourier transform is computationally expensive. In this paper, we present a novel, fast algorithm for directly computing the fractional period spectrum (FPS) of time series. The algorithm is based on the periodic distribution of signal strength at periodic positions of the time series. We provide theoretical analysis, deduction, and special techniques for reducing the computational costs of the algorithm. The analysis of the computational complexity of the algorithm shows that the algorithm is much faster than traditional Fourier transform. Our algorithm can be applied directly in computing fractional periods in time series from a broad of research fields. The computer programs of the FPS algorithm are available at https://github.com/cyinbox/FPS.
\end{abstract}
\end{frontmatter}
\section{Introduction}
\label{sec1} 
Signal processing approaches have been widely applied in the periodicity analysis of time series, and thus becoming a major research tool for bioinformatics \citep{anastassiou2001genomic,chen2003will}. The main requisite in applying signal processing to symbolic sequences is mapping the sequences onto numerical time series. For example, we previously presented a numerical representation of DNA sequences without degeneracy \citep{yau2003dna}. After numerical mapping, mathematical methods can be employed to study features, structures, and functions of the symbolic sequences. The most common signal processing approach is Fourier transform \citep{welch1967use}, which has been widely used to study periodicity and repetitive regions in symbolic DNA sequences \citep{silverman1986measure,anastassiou2001genomic}, as well as in genome comparison \citep{yin2014measure,yin2015improved}. We have studied the mathematical properties of the Fourier spectrum for symbolic sequences \citep{wang2012Some,wang2014SNR}. For a review of mathematical methods for the study of biological sequences, one may refer to our book \citep{wang2013numerical}.

Periods in biological sequences can be categorized into two types: integer periods and fractional periods. For integer periods, the 3-base periodicity of protein-coding regions is often used in gene finding \citep{tiwari1997prediction,yin2005fourier,yin2007prediction,yin2014representation}; 2-base periodicity was found in introns of genomes \citep{arques1987periodicities}. 2-period exists in protein sequence regions for beta-sheet structures. As indicated by their name, fractional periods do not correspond to an integer number of cycles in numerical series. Fractional periods are prevalent and important in structures and functions of protein sequences and genomes. For example, the 3.6-period in protein sequences determines the $\alpha$-helix secondary structure \citep{EISENBERG1984,gruber2005repper,leonov2005periodicity,yin2017coevolution}. Strong 10.4- or 10.5-base periodicity in genomes are associated with nucleosomes \citep{trifonov19983,salih2015strong}. The 6.5-base periodicity is present in \textit{C. elegans} introns \citep{messaoudi2013detection}. Fractional period spectrum may offer high resolution and precise features of sequences. However, identifying fractional period features in a large genome is challenging. The demand of the periodicity research of DNA and protein sequences motivates us to investigate advanced techniques for fractional periodicity analysis of time series, which correspond to biological symbolic sequences.

With the quickly growing volume of genomic sequences of human and model organisms, there is a strong demand to characterize not only integer periods but also fractional periods in genomic and protein sequences. We previously proposed a method, periodic power spectrum (PPS), which can directly compute Fourier power spectrum based on periodic distributions of signal strength on periodic positions \citep {wang2012Some,yin2016periodic}. The main advantage the PPS method is that it can avoid spectral leakage and reduce background noise, which both appear in power spectrum in Fourier transform. Thus, the PPS method can capture all latent integer periodicities in DNA sequences. We have applied this method in detection of latent periodicities in different genome elements, including exons and microsatellite DNA sequences. 

This paper extends the PPS method to suggest a novel method for fast and directly calculating Fourier spectrum of a fractional period. This method will be also helpful in directly computing the fractional period spectrum for any time series from science and technology fields outside of bioinformatics. 

\section{Methods Deduced}
\subsection{Fourier transform and congruence derivative sequence}
For a numerical sequence $x$ of length $m$, consisting of $x_0 ,x_1 , \cdots ,x_{m - 1}$, its discrete Fourier transform (DFT) at frequency $k$ is defined as \citep{welch1967use}
\begin{equation}
X(k) = \sum\limits_{j = 0}^{m - 1} {x_j e^{ - i2\pi kj/m} } ,i = \sqrt { - 1,} {\kern 1pt} {\kern 1pt} {\kern 1pt} k = 0,1, \ldots ,m - 1
 \end{equation}
 
And its Fourier power spectrum at frequency $k$ is defined as \citep{welch1967use}
 \begin{equation}
\begin{gathered}
  PS(k) = X^* (k)X(k) \hfill \\
   = (\sum\limits_{j = 0}^{m - 1} {x_j e^{ - i2\pi jk/m} } )^* (\sum\limits_{j = 0}^{m - 1} {x_j e^{ - i2\pi jk/m} } ) \hfill \\
  i = \sqrt { - 1,} k = 0,1, \ldots ,m - 1 \hfill \\ 
\end{gathered} 
\end{equation}
where $*$ indicates the complex conjugate. Because our method is applied to researching the periodicity of sequence by using Fourier spectrum, for convenience, the frequency $k/m$ is called period $m/k$. In this paper, we directly define the DFT spectrum by period instead of frequency.
 
We want to compute Fourier spectrum at period $l/k$, $l > k$ and $l$ ,$k$ are smaller than $m$. Using traditional formula of Fourier transform, we need to calculate Fourier transform 
 
 \begin{equation}
 X(k) = \sum\limits_{j = 0}^{m - 1} {x_j e^{\frac{{ - i2\pi jk}}
 {l}} }
 \end{equation}
 
Then, its spectrum at period $l/k$ or at frequency $k$ is
 \begin{equation}
PS(k) = (\sum\limits_{j = 0}^{m - 1} {x_j e^{\frac{{ - i2\pi jk}}
{l}} )^* } (\sum\limits_{j = 0}^{m - 1} {x_j e^{\frac{{ - i2\pi jk}}
{l}} } )
\end{equation}
We will introduce a fast method for directly computing the Fourier spectrum at period $l/k$. We previously demonstrate that the Fourier power spectrum of a numerical sequence is determined by the distribution of signal strength at periodic positions in this sequence \citep{wang2012Some,yin2016periodic}. The distribution can be represented and measured by the congruence derivative sequence of the original sequence, which is defined as follows.

 \begin{definition} For a real number sequence $x$ of length $m$, if two positive integers $n$ and $l$ satisfy $m = nl$, then the congruence derivative sequence, $y$, of $x$ has a length of $l$ and its element is defined as follows \\
 	 \begin{equation}
 	y_t  = \sum\limits_{j = 0}^{n - 1} {x_{jl + t} } ,t = 0,1, \ldots l - 1
 	 \end{equation}
 \end{definition}
 
 The following deduced theorems indicate that the congruence derivative sequence of length $l$ of the numerical sequence $x$ can be used to compute periodic power spectrum at period $l$. The Fourier transform of the congruence derivative sequence $y$ is 
 \begin{equation}
Y(k) = \sum\limits_{t = 0}^{l - 1} {y_t } e^{ - i2\pi tk/l} ,k = 0,1, \ldots l - 1
\end{equation}

The relationship between Fourier transform of the original sequence and its congruence derivative sequence was introduced in \citep{wang2012Some,yin2016periodic}. The DFT of sequence $x$ at frequency $n = m/l$ is the same as the DFT of the congruence derivative sequence $y$ at frequency $l$.  
\begin{equation}
X(n) = X(m/l) = Y(1)
\end{equation}
This formula is the mathematical ground of the PPS method. The following theorem extends the conclusion to a general case.
  \begin{theorem}
  	For a real number sequence $x$ of length $m$, suppose $n = m/l$, then the DFT of the numerical sequence $x$ at frequency $kn$ is 
  \begin{equation}
  	X(kn) = Y(k),0 \le k \le l - 1
  \end{equation}
 \end{theorem}
  
  \begin{proof}
  We know \\
  $\begin{gathered}
    Y(k) = \sum\limits_{t = 0}^{l - 1} {y_t } e^{ - i2\pi tk/l}  = \sum\limits_{t = 0}^{l - 1} {\sum\limits_{j = 0}^{n - 1} {x_{jl + t} } e^{ - i2\pi tk/l} }  = \sum\limits_{t = 0}^{l - 1} {\sum\limits_{j = 0}^{n - 1} {x_{jl + t} } e^{ - i2\pi (jl + t)k/l} }  = \sum\limits_{r = 0}^{m - 1} {x_r } e^{ - i2\pi rkn/m}  \hfill \\
     = X(kn) \hfill \\ 
  \end{gathered}$ 
\end{proof}

The proof of Theorem 2.1 indicates that the property of conjugate symmetry in DFT of a real sequence is preserved. We thus have the following corollary.
 \begin{cor}
  	\begin{equation}
  X(kn) = X^* ((l - k)n),1 \leqslant k \leqslant l - 1
  \end{equation}
  \end{cor}
  
  According to formula 2.2, Fourier power spectrum of sequence $x$ at frequency $kn$ is defined as
  \begin{equation}
  PS(kn) = X^* (kn)X(kn) = X^* (km/l)X(km/l) = Y^* (k)Y(k).
  \end{equation}
 
 \begin{definition}
If $m = nl$, the DFT spectrum of numerical sequence $x$ at fractional period $l/k $ is named the fractional period spectrum (FPS) and is written as
$FPS(l/k) = Y^* (k)Y(k),1 \le k \le l - 1$
\end{definition}

\begin{definition}
For a real number sequence of length $l$, $y_t,t = 0,1, \ldots ,l - 1$, its self $q$-shift summation is defined by 
\begin{equation} 
z_{l,q}  = \sum\limits_{t = 0}^{l - 1} {y_t y_{t + q} } 
\end{equation}
with $t+q$ taken modulo $l$, $0\leqslant q < l$.
\end{definition}

If we write sequence $y_t$, $t = 0,1, \ldots ,l - 1$, as a vector, $y^T  = [y_0 ,y_1 , \cdots ,y_{l - 1} ]$ , then $z_{l,q} $ can be realized by the autocorrelation of vector $y$. For a special case, $z_{l,0}$ is equal to the inner product of vector $y$, or $z_{l,0}  = y^Ty$.

From the definition of DFT of the congruence derivative sequence (formula 2.6), we have
\begin{equation}
Y(k) = \sum\limits_{t = 0}^{l - 1} {y_t e^{ - i2\pi tk/l}  = } \sum\limits_{t = 0}^{l - 1} {y_t \cos (2\pi tk/l) - i} \sum\limits_{t = 0}^{l - 1} {y_t \sin (2\pi tk/l)} 
\end{equation}

From the definition of $FPS$, $FPS(l/k) = Y^* (k)Y(k)$, we notice that
\begin{equation}
FPS(l/k) = (\sum\limits_{t = 0}^{l - 1} {y_t \cos (2\pi tk/l))^2  + (} \sum\limits_{t = 0}^{l - 1} {y_t \sin (2\pi tk/l)} )^2 
\end{equation}

which is real quadratic form of $l$ variables, $y_0 ,y_1 , \ldots ,y_{l - 1}$, can be written as 
\begin{equation}
FPS(l/k) = y^t Ay
\end{equation}
where the coefficient matrix $A=a_{rs}$ is of quadratic form, in which
\begin{equation}
\begin{gathered}
  a_{rs}  = \cos ((r - l)\frac{{2k\pi }}
{l})\cos ((s - l)\frac{{2k\pi }}
{l}) + \sin ((r - l)\frac{{2k\pi }}
{l})\sin ((s - l)\frac{{2k\pi }}
{l}) \hfill \\
  r,s = 1,2, \ldots ,l \hfill \\ 
\end{gathered} 
\end{equation}

After reorganizing formula (2.15), we obtain the following theorem about the coefficient matrix $A$.
   \begin{theorem}  
   The coefficient matrix of the quadratic form, $A$, is a symmetric matrix.
   \end{theorem}

 \begin{proof}
   According to some fundamental identities of trigonometry, 
\[
a_{rs}  = \cos ((r - l)\frac{{2k\pi }}
{l})\cos ((s - l)\frac{{2k\pi }}
{l}) + \sin ((r - l)\frac{{2k\pi }}
{l})\sin ((s - l)\frac{{2k\pi }}
{l})
\]
would be equal to 
$\cos ((r - l - (s - l))\frac{{2k\pi }}
{l}) = \cos ((r - s)\frac{{2k\pi }}
{l})$
, \textit{i.e.}, 
$a_{rs}  = \cos ((r - s)\frac{{2k\pi }}
{l})$.

 It is clear that $a_{rs} = a_{sr}$, so the coefficient matrix $A$ is symmetric. 
 \end{proof}
  
 \begin{theorem} For a real number sequence $x$ of length $m$, if $m=nl$ and its congruence derivative sequence is $y_t, t = 0,1, \cdots ,l - 1$, then $FPS(l/k)$ of the sequence $x$ can be expressed as follows.
   
   If $l$ is an odd number,
  \[
  FPS(l/k) = z_{l,0}  + 2\sum\limits_{q = 1}^{\frac{{l - 1}}
  {2}} {z_{l,q} \cos ((q)} \frac{{2k\pi }}
  {l}).
  \]
   
   If $l$ is an even number,
   \[
   FPS(l/k) = z_{l,0}  + 2\sum\limits_{q = 1}^{\frac{l}
   {2} - 1} {z_{l,q} \cos ((q)\frac{{2k\pi }}
   {l}) + } 2\cos ((\frac{l}
   {2})\frac{{2k\pi }}
   {l})\sum\limits_{t = 0}^{\frac{l}
   {2} - 1} {y_t y_{t + \frac{l}
   {2}} } 
   \]
   \end{theorem}
\begin{proof}

It is clear the diagonal elements of the coefficient matrix $A$ are unit elements. Due to the symmetry of matrix $A$, we need to check the lower triangular matrix of $A$ under the diagonal of $A$, which is sufficient for proof. 

The element of the lower triangular matrix of $A$, $a_{rs}  = \cos ((r - s)\frac{{2k\pi }}{l}),r = 2,3, \ldots ,l,r > s$, also has some kind of symmetry. The result is as follows. 

When $l$ is an odd number, $a_{21}  = a_{32}  = a_{43}  =  \ldots  \ldots  = a_{l(l - 1)}  = \cos ((1)\frac{{2k\pi }}{l}) $ and $a_{21} ,a_{32} , \ldots  \ldots ,a_{l(l - 1)} $ also are equal to $a_{l1}$. Because $a_{rs}  = a_{l - t{\kern 1pt} {\kern 1pt} {\kern 1pt} {\kern 1pt} r - s - t} ,r > s,0 \leqslant t < r - s$, similarly, $a_{31}  = a_{42}  =  \ldots  \ldots  = a_{l(l - 2)}  = a_{(l - 1)1}  = a_{l2}  = \cos ((2)\frac{{2k\pi }}
{l}) \ldots  \ldots$, then $a_{(l + 1)/21}  = a_{(l + 3)/22}  =  \ldots  \ldots , = a_{l(l + 1)/2}  = a_{(l + 3)/21}  =  \ldots  \ldots ,a_{l(l - 1)/2}  = \cos (((l - 1)/2)\frac{{2k\pi }}
{l}).$

When $l$ is an even number, we have $a_{21}  = a_{32}  = a_{43}  =  \ldots  \ldots  = a_{l(l - 1)} = a_{l1}  = \cos ((1)\frac{{2k\pi }}
{l})$
and   $a_{31}  = a_{42}  =  \ldots  \ldots  = a_{l(l - 2)} = a_{(l - 1)1}  = a_{l2}  = \cos ((2)\frac{{2k\pi }}
{l}).$ In general, $a_{l/21}  = a_{(l/2 + 1)2}  =  \ldots  \ldots , = a_{l(l/2 + 1)}  = a_{(l/2 + 2)1}  = a_{(l/2 + 3)2}  =  \ldots  \ldots , = a_{l(l/2 - 1)}  = \cos ((l/2 - 1)\frac{{2k\pi }}
 {l}),$ and $a_{(l/2 + 1)1}  = a_{(l/2 + 2)2}  =  \ldots  \ldots , = a_{ll/2}  = \cos ((l/2)\frac{{2k\pi }}
{l}).$

We then substitute those results into the coefficient matrix of the quadratic form $A$ for the $FPS(l/k)$. The following formulas can be obtained.

If $l$ is an odd number,
\begin{equation}
\begin{gathered}
  FPS(l/k) = y_0^2  + y_1^2  + y_2^2  +  \cdots  + y_{l - 1}^2  \hfill \\
  {\kern 1pt} {\kern 1pt} {\kern 1pt} {\kern 1pt} {\kern 1pt} {\kern 1pt} {\kern 1pt} {\kern 1pt} {\kern 1pt} {\kern 1pt} {\kern 1pt} {\kern 1pt} {\kern 1pt} {\kern 1pt} {\kern 1pt} {\kern 1pt} {\kern 1pt} {\kern 1pt} {\kern 1pt} {\kern 1pt} {\kern 1pt} {\kern 1pt} {\kern 1pt} {\kern 1pt} {\kern 1pt} {\kern 1pt} {\kern 1pt} {\kern 1pt} {\kern 1pt} {\kern 1pt} {\kern 1pt} {\kern 1pt} {\kern 1pt} {\kern 1pt} {\kern 1pt} {\kern 1pt} {\kern 1pt} {\kern 1pt} {\kern 1pt} {\kern 1pt} {\kern 1pt} {\kern 1pt} {\kern 1pt} {\kern 1pt} {\kern 1pt} {\kern 1pt} {\kern 1pt} {\kern 1pt} {\kern 1pt} {\kern 1pt} {\kern 1pt} {\kern 1pt} {\kern 1pt} {\kern 1pt} {\kern 1pt} {\kern 1pt} {\kern 1pt}  + 2\cos ((1)\frac{{2k\pi }}
{l})(y_0 y_1  + y_1 y_2  + y_2 y_3  + y_3 y_4  \cdots y_{l - 1} y_0 ) \hfill \\
  {\kern 1pt} {\kern 1pt} {\kern 1pt} {\kern 1pt} {\kern 1pt} {\kern 1pt} {\kern 1pt} {\kern 1pt} {\kern 1pt} {\kern 1pt} {\kern 1pt} {\kern 1pt} {\kern 1pt} {\kern 1pt} {\kern 1pt} {\kern 1pt} {\kern 1pt} {\kern 1pt} {\kern 1pt} {\kern 1pt} {\kern 1pt} {\kern 1pt} {\kern 1pt} {\kern 1pt} {\kern 1pt} {\kern 1pt} {\kern 1pt} {\kern 1pt} {\kern 1pt} {\kern 1pt} {\kern 1pt} {\kern 1pt} {\kern 1pt} {\kern 1pt} {\kern 1pt} {\kern 1pt} {\kern 1pt} {\kern 1pt} {\kern 1pt} {\kern 1pt} {\kern 1pt} {\kern 1pt} {\kern 1pt} {\kern 1pt} {\kern 1pt} {\kern 1pt} {\kern 1pt} {\kern 1pt} {\kern 1pt} {\kern 1pt} {\kern 1pt} {\kern 1pt} {\kern 1pt} {\kern 1pt} {\kern 1pt} {\kern 1pt} {\kern 1pt} {\kern 1pt}  + 2\cos ((2)\frac{{2k\pi }}
{l})(y_0 y_2  + y_1 y_3  + y_2 y_4  +  + y_3 y_5  \cdots y_{l - 2} y_0  + y_{l - 1} y_1 ) \hfill \\
  {\kern 1pt} {\kern 1pt} {\kern 1pt} {\kern 1pt} {\kern 1pt} {\kern 1pt} {\kern 1pt} {\kern 1pt} {\kern 1pt} {\kern 1pt} {\kern 1pt} {\kern 1pt} {\kern 1pt} {\kern 1pt} {\kern 1pt} {\kern 1pt} {\kern 1pt} {\kern 1pt} {\kern 1pt} {\kern 1pt} {\kern 1pt} {\kern 1pt} {\kern 1pt} {\kern 1pt} {\kern 1pt} {\kern 1pt} {\kern 1pt} {\kern 1pt} {\kern 1pt} {\kern 1pt} {\kern 1pt} {\kern 1pt} {\kern 1pt} {\kern 1pt} {\kern 1pt} {\kern 1pt} {\kern 1pt} {\kern 1pt} {\kern 1pt} {\kern 1pt} {\kern 1pt} {\kern 1pt} {\kern 1pt} {\kern 1pt} {\kern 1pt} {\kern 1pt} {\kern 1pt} {\kern 1pt} {\kern 1pt} {\kern 1pt} {\kern 1pt} {\kern 1pt} {\kern 1pt} {\kern 1pt} {\kern 1pt} {\kern 1pt} {\kern 1pt}  + 2\cos ((3)\frac{{2k\pi }}
{l})(y_0 y_3  + y_1 y_4  + y_2 y_5  +  + y_3 y_6  \cdots y_{l - 3} y_0  + y_{l - 2} y_1  + y_{l - 1} y_2 ) \hfill \\
  {\kern 1pt} {\kern 1pt} {\kern 1pt} {\kern 1pt} {\kern 1pt} {\kern 1pt} {\kern 1pt} {\kern 1pt} {\kern 1pt} {\kern 1pt} {\kern 1pt} {\kern 1pt} {\kern 1pt} {\kern 1pt} {\kern 1pt} {\kern 1pt} {\kern 1pt} {\kern 1pt} {\kern 1pt} {\kern 1pt} {\kern 1pt} {\kern 1pt} {\kern 1pt} {\kern 1pt} {\kern 1pt} {\kern 1pt} {\kern 1pt} {\kern 1pt} {\kern 1pt} {\kern 1pt} {\kern 1pt} {\kern 1pt} {\kern 1pt} {\kern 1pt} {\kern 1pt} {\kern 1pt} {\kern 1pt} {\kern 1pt} {\kern 1pt} {\kern 1pt} {\kern 1pt} {\kern 1pt} {\kern 1pt} {\kern 1pt} {\kern 1pt} {\kern 1pt} {\kern 1pt} {\kern 1pt} {\kern 1pt} {\kern 1pt} {\kern 1pt} {\kern 1pt} {\kern 1pt} {\kern 1pt} {\kern 1pt} {\kern 1pt} {\kern 1pt} {\kern 1pt} {\kern 1pt} {\kern 1pt} {\kern 1pt}  \cdots  \hfill \\
  {\kern 1pt} {\kern 1pt} {\kern 1pt} {\kern 1pt} {\kern 1pt} {\kern 1pt} {\kern 1pt} {\kern 1pt} {\kern 1pt} {\kern 1pt} {\kern 1pt} {\kern 1pt} {\kern 1pt} {\kern 1pt} {\kern 1pt} {\kern 1pt} {\kern 1pt} {\kern 1pt} {\kern 1pt} {\kern 1pt} {\kern 1pt} {\kern 1pt} {\kern 1pt} {\kern 1pt} {\kern 1pt} {\kern 1pt} {\kern 1pt} {\kern 1pt} {\kern 1pt} {\kern 1pt} {\kern 1pt} {\kern 1pt} {\kern 1pt} {\kern 1pt} {\kern 1pt} {\kern 1pt} {\kern 1pt} {\kern 1pt} {\kern 1pt} {\kern 1pt} {\kern 1pt} {\kern 1pt} {\kern 1pt} {\kern 1pt} {\kern 1pt} {\kern 1pt} {\kern 1pt} {\kern 1pt} {\kern 1pt} {\kern 1pt} {\kern 1pt} {\kern 1pt} {\kern 1pt} {\kern 1pt} {\kern 1pt}  + 2\cos (((l - 1)/2)\frac{{2k\pi }}
{l})(y_0 y_{(l - 1)/2}  + y_1 y_{(l + 1)/2}  + y_2 y_{(l + 3)/2}  +  \hfill \\
  {\kern 1pt} {\kern 1pt} {\kern 1pt} {\kern 1pt} {\kern 1pt} {\kern 1pt} {\kern 1pt} {\kern 1pt} {\kern 1pt} {\kern 1pt} {\kern 1pt} {\kern 1pt} {\kern 1pt} {\kern 1pt} {\kern 1pt} {\kern 1pt} {\kern 1pt} {\kern 1pt} {\kern 1pt} {\kern 1pt} {\kern 1pt} {\kern 1pt} {\kern 1pt} {\kern 1pt} {\kern 1pt} {\kern 1pt} {\kern 1pt} {\kern 1pt} {\kern 1pt} {\kern 1pt} {\kern 1pt} {\kern 1pt} {\kern 1pt} {\kern 1pt} {\kern 1pt} {\kern 1pt} {\kern 1pt} {\kern 1pt} {\kern 1pt} {\kern 1pt} {\kern 1pt} {\kern 1pt} {\kern 1pt} {\kern 1pt} {\kern 1pt} {\kern 1pt} {\kern 1pt} {\kern 1pt} {\kern 1pt} {\kern 1pt} {\kern 1pt} {\kern 1pt} {\kern 1pt} {\kern 1pt} {\kern 1pt} {\kern 1pt} {\kern 1pt} {\kern 1pt} {\kern 1pt} {\kern 1pt} {\kern 1pt} {\kern 1pt} {\kern 1pt} {\kern 1pt} {\kern 1pt} {\kern 1pt} {\kern 1pt} {\kern 1pt} {\kern 1pt} {\kern 1pt} {\kern 1pt} {\kern 1pt} {\kern 1pt} {\kern 1pt} {\kern 1pt} {\kern 1pt} {\kern 1pt} {\kern 1pt} {\kern 1pt} {\kern 1pt} {\kern 1pt} {\kern 1pt} {\kern 1pt} {\kern 1pt} {\kern 1pt} {\kern 1pt} {\kern 1pt} {\kern 1pt} {\kern 1pt} {\kern 1pt} {\kern 1pt} {\kern 1pt} {\kern 1pt} {\kern 1pt} {\kern 1pt} {\kern 1pt} {\kern 1pt} {\kern 1pt} {\kern 1pt} {\kern 1pt} {\kern 1pt} {\kern 1pt} {\kern 1pt} {\kern 1pt} {\kern 1pt} {\kern 1pt} {\kern 1pt} {\kern 1pt} {\kern 1pt} {\kern 1pt} {\kern 1pt} {\kern 1pt} {\kern 1pt} {\kern 1pt} {\kern 1pt} {\kern 1pt} {\kern 1pt} {\kern 1pt} {\kern 1pt} {\kern 1pt}  \cdots y_{(l - (l - 1)/2)} y_{_0 }  +  \cdots  + y_{l - 1} y_{(l - 3)/2} ). \hfill \\ 
\end{gathered} 
\end{equation}

If $l$ is an even number,
\begin{equation}
\begin{gathered}
  FPS(l/k) = y_0^2  + y_1^2  + y_2^2  +  \cdots  + y_{l - 1}^2  \hfill \\
  {\kern 1pt} {\kern 1pt} {\kern 1pt} {\kern 1pt} {\kern 1pt} {\kern 1pt} {\kern 1pt} {\kern 1pt} {\kern 1pt} {\kern 1pt} {\kern 1pt} {\kern 1pt} {\kern 1pt} {\kern 1pt} {\kern 1pt} {\kern 1pt} {\kern 1pt} {\kern 1pt} {\kern 1pt} {\kern 1pt} {\kern 1pt} {\kern 1pt} {\kern 1pt} {\kern 1pt} {\kern 1pt} {\kern 1pt} {\kern 1pt} {\kern 1pt} {\kern 1pt} {\kern 1pt} {\kern 1pt} {\kern 1pt} {\kern 1pt} {\kern 1pt} {\kern 1pt} {\kern 1pt} {\kern 1pt} {\kern 1pt} {\kern 1pt} {\kern 1pt} {\kern 1pt} {\kern 1pt} {\kern 1pt} {\kern 1pt} {\kern 1pt} {\kern 1pt} {\kern 1pt} {\kern 1pt} {\kern 1pt} {\kern 1pt} {\kern 1pt} {\kern 1pt} {\kern 1pt} {\kern 1pt} {\kern 1pt} {\kern 1pt} {\kern 1pt}  + 2\cos ((1)\frac{{2k\pi }}
{l})(y_0 y_1  + y_1 y_2  + y_2 y_3  + y_3 y_4  \cdots y_{l - 1} y_0 ) \hfill \\
  {\kern 1pt} {\kern 1pt} {\kern 1pt} {\kern 1pt} {\kern 1pt} {\kern 1pt} {\kern 1pt} {\kern 1pt} {\kern 1pt} {\kern 1pt} {\kern 1pt} {\kern 1pt} {\kern 1pt} {\kern 1pt} {\kern 1pt} {\kern 1pt} {\kern 1pt} {\kern 1pt} {\kern 1pt} {\kern 1pt} {\kern 1pt} {\kern 1pt} {\kern 1pt} {\kern 1pt} {\kern 1pt} {\kern 1pt} {\kern 1pt} {\kern 1pt} {\kern 1pt} {\kern 1pt} {\kern 1pt} {\kern 1pt} {\kern 1pt} {\kern 1pt} {\kern 1pt} {\kern 1pt} {\kern 1pt} {\kern 1pt} {\kern 1pt} {\kern 1pt} {\kern 1pt} {\kern 1pt} {\kern 1pt} {\kern 1pt} {\kern 1pt} {\kern 1pt} {\kern 1pt} {\kern 1pt} {\kern 1pt} {\kern 1pt} {\kern 1pt} {\kern 1pt} {\kern 1pt} {\kern 1pt} {\kern 1pt} {\kern 1pt} {\kern 1pt} {\kern 1pt}  + 2\cos ((2)\frac{{2k\pi }}
{l})(y_0 y_2  + y_1 y_3  + y_2 y_4  +  + y_3 y_5  \cdots y_{l - 2} y_0  + y_{l - 1} y_1 ) \hfill \\
  {\kern 1pt} {\kern 1pt} {\kern 1pt} {\kern 1pt} {\kern 1pt} {\kern 1pt} {\kern 1pt} {\kern 1pt} {\kern 1pt} {\kern 1pt} {\kern 1pt} {\kern 1pt} {\kern 1pt} {\kern 1pt} {\kern 1pt} {\kern 1pt} {\kern 1pt} {\kern 1pt} {\kern 1pt} {\kern 1pt} {\kern 1pt} {\kern 1pt} {\kern 1pt} {\kern 1pt} {\kern 1pt} {\kern 1pt} {\kern 1pt} {\kern 1pt} {\kern 1pt} {\kern 1pt} {\kern 1pt} {\kern 1pt} {\kern 1pt} {\kern 1pt} {\kern 1pt} {\kern 1pt} {\kern 1pt} {\kern 1pt} {\kern 1pt} {\kern 1pt} {\kern 1pt} {\kern 1pt} {\kern 1pt} {\kern 1pt} {\kern 1pt} {\kern 1pt} {\kern 1pt} {\kern 1pt} {\kern 1pt} {\kern 1pt} {\kern 1pt} {\kern 1pt} {\kern 1pt} {\kern 1pt} {\kern 1pt} {\kern 1pt} {\kern 1pt}  + 2\cos ((3)\frac{{2k\pi }}
{l})(y_0 y_3  + y_1 y_4  + y_2 y_5  +  + y_3 y_6  \cdots y_{l - 3} y_0  + y_{l - 2} y_1  + y_{l - 1} y_2 ) \hfill \\
  {\kern 1pt} {\kern 1pt} {\kern 1pt} {\kern 1pt} {\kern 1pt} {\kern 1pt} {\kern 1pt} {\kern 1pt} {\kern 1pt} {\kern 1pt} {\kern 1pt} {\kern 1pt} {\kern 1pt} {\kern 1pt} {\kern 1pt} {\kern 1pt} {\kern 1pt} {\kern 1pt} {\kern 1pt} {\kern 1pt} {\kern 1pt} {\kern 1pt} {\kern 1pt} {\kern 1pt} {\kern 1pt} {\kern 1pt} {\kern 1pt} {\kern 1pt} {\kern 1pt} {\kern 1pt} {\kern 1pt} {\kern 1pt} {\kern 1pt} {\kern 1pt} {\kern 1pt} {\kern 1pt} {\kern 1pt} {\kern 1pt} {\kern 1pt} {\kern 1pt} {\kern 1pt} {\kern 1pt} {\kern 1pt} {\kern 1pt} {\kern 1pt} {\kern 1pt} {\kern 1pt} {\kern 1pt} {\kern 1pt} {\kern 1pt} {\kern 1pt} {\kern 1pt} {\kern 1pt} {\kern 1pt} {\kern 1pt} {\kern 1pt} {\kern 1pt} {\kern 1pt} {\kern 1pt} {\kern 1pt} {\kern 1pt}  \cdots  \hfill \\
  {\kern 1pt} {\kern 1pt} {\kern 1pt} {\kern 1pt} {\kern 1pt} {\kern 1pt} {\kern 1pt} {\kern 1pt} {\kern 1pt} {\kern 1pt} {\kern 1pt} {\kern 1pt} {\kern 1pt} {\kern 1pt} {\kern 1pt} {\kern 1pt} {\kern 1pt} {\kern 1pt} {\kern 1pt} {\kern 1pt} {\kern 1pt} {\kern 1pt} {\kern 1pt} {\kern 1pt} {\kern 1pt} {\kern 1pt} {\kern 1pt} {\kern 1pt} {\kern 1pt} {\kern 1pt} {\kern 1pt} {\kern 1pt} {\kern 1pt} {\kern 1pt} {\kern 1pt} {\kern 1pt} {\kern 1pt} {\kern 1pt} {\kern 1pt} {\kern 1pt} {\kern 1pt} {\kern 1pt} {\kern 1pt} {\kern 1pt} {\kern 1pt} {\kern 1pt} {\kern 1pt} {\kern 1pt} {\kern 1pt} {\kern 1pt} {\kern 1pt} {\kern 1pt} {\kern 1pt} {\kern 1pt} {\kern 1pt}  + 2\cos ((l/2 - 1)\frac{{2k\pi }}
{l})(y_0 y_{l/2 - 1}  + y_1 y_{l/2}  + y_2 y_{l/2 + 1}  +  \cdots y_{l - 3} y_{l/2 - 2}  + y_{l - 2} y_{l/2 - 1}  + y_{l - 1} y_{l/2} ) \hfill \\
  {\kern 1pt} {\kern 1pt} {\kern 1pt} {\kern 1pt} {\kern 1pt} {\kern 1pt} {\kern 1pt} {\kern 1pt} {\kern 1pt} {\kern 1pt} {\kern 1pt} {\kern 1pt} {\kern 1pt} {\kern 1pt} {\kern 1pt} {\kern 1pt} {\kern 1pt} {\kern 1pt} {\kern 1pt} {\kern 1pt} {\kern 1pt} {\kern 1pt} {\kern 1pt} {\kern 1pt} {\kern 1pt} {\kern 1pt} {\kern 1pt} {\kern 1pt} {\kern 1pt} {\kern 1pt} {\kern 1pt} {\kern 1pt} {\kern 1pt} {\kern 1pt} {\kern 1pt} {\kern 1pt} {\kern 1pt} {\kern 1pt} {\kern 1pt} {\kern 1pt} {\kern 1pt} {\kern 1pt} {\kern 1pt} {\kern 1pt} {\kern 1pt} {\kern 1pt} {\kern 1pt} {\kern 1pt} {\kern 1pt} {\kern 1pt} {\kern 1pt} {\kern 1pt} {\kern 1pt} {\kern 1pt} {\kern 1pt}  + 2\cos ((l/2)\frac{{2k\pi }}
{l})(y_0 y_{l/2}  + y_1 y_{l/2 + 1}  + y_2 y_{l/2 + 2}  +  \cdots y_{l/2 - 1} y_{l - 1} ). \hfill \\ 
\end{gathered} 
\end{equation}
The last term in equation (2.17), $\cos ((l/2)\frac{{2k\pi }}
{l})(y_0 y_{l/2}  + y_1 y_{l/2 + 1}  + y_2 y_{l/2 + 2}  +  \cdots  \cdots  + y_{l/2 - 1} y_{l - 1} )$ can be written as $\cos ((l/2)\frac{{2k\pi }}
 {l})\sum\limits_{t = 0}^{l/2 - 1} {y_t y_{t + l/2} } $. Noticing the definition of $z_{l,q}$, the formulas (2.16) and (2.17) may be written as follows.

If $l$ is an odd number,
\begin{equation}	
 FPS(l/k) = z_{l,0}  + 2\sum\limits_{q = 1}^{\frac{{l - 1}}
 {2}} {z_{l,q} \cos ((q)} \frac{{2k\pi }}
 {l})
\end{equation}
   	
   If $l$ is an even number,
   \begin{equation}
    FPS(l/k) = z_{l,0}  + 2\sum\limits_{q = 1}^{\frac{l}
    {2} - 1} {z_{l,q} \cos ((q)\frac{{2k\pi }}
    {l}) + } 2\cos ((\frac{l}
    {2})\frac{{2k\pi }}
    {l})\sum\limits_{t = 0}^{\frac{l}
    {2} - 1} {y_t y_{t + \frac{l}
    {2}} } 
    \end{equation}	
\end{proof}
 
Based on Theorem 2.3 and Corollary 2.1, $FPS(l/k)$ may be re-written by following theorem, which may improve the technique for computing $FPS(l/k)$. 

\begin{theorem} 
	\begin{equation}
FPS(l/k) = FPS(l/(l - k),)1 \leqslant k \leqslant l - 1
  \end{equation}
\end{theorem}

According to Theorem 2.4, to compute all of the $FPS(l/k)$, $1 \leqslant k \leqslant l - 1$, we only require about half of computing costs. 

For $k = 1$, noticing formulas (2.18) and (2.19) and the symmetry of $FPS(l/k)$, we need to calculate $(l-1)/2$ coefficients of quadratics $y_0 ,y_1 , \ldots ,y_{l - 1}$ when $l$ is an odd number; and $l/2$ coefficients when $l$ is an even number. 

It is valuable to notice on the property of a cosine function,
\begin{equation}   
\cos ((q)\frac{{2k\pi }}
{l}) = \cos ((q)\frac{{2(l - k)\pi }}
{l}),0 \leqslant k \leqslant l - 1.
\end{equation}
Certainly, it is the foundation of Theorem 2.4 and an important formula to prove the following theorem.
	
\begin{theorem} The coefficient matrix of quadratic form $A$ when $k=1$ is sufficient for computing $FPS(l/k)$ when $k \ne 1$.
\end{theorem}

\begin{proof}
Without loss of generality, we show the case when $l$ is an odd number. For a fixed $k$, $k \ne 1$, from Theorem 2.4, the range of $k$ should be in $2 \leqslant k \leqslant (l - 1)/2$ and the $(l-1)/2$ coefficients of quadratics of $y_0 ,y_1 , \cdots ,y_{l - 1} $ are $\cos ((t)\frac{{2k\pi }}{l}),t = 1,2, \ldots ,\frac{{l - 1}}{2}$
   
For $k=1$, we obtain the coefficients 
$\cos ((q)\frac{{2\pi }}{l}),q = 1,2, \ldots ,\frac{{l - 1}}{2}$.
   
 If the $(l-1)/2$ coefficients, 
 $ \cos ((t)\frac{{2k\pi }}
   {l}),t = 1,2, \cdots ,(l - 1)/2$ can be represented by $ \cos((q)\frac{{2\pi }}
      {l}),q = 1,2, \cdots ,(l - 1)/2$, then, the proposition is proved. 
   
 It is known that the interval of $k$ is $[2, (l-1)/2]$, so, for $t = 1, \cos ((k)\frac{{2\pi }}{l})$, the first coefficient of $FPS(l/k)$, is one of $\cos ((q)\frac{{2\pi }}{l}),q = 1,2, \ldots ,(l - 1)/2$. When $t \ne 1$, we may have one of two results: if $tk$ in the interval $[1, (l-1)/2]$, it is obvious that $\cos ((t)\frac{{2k\pi }}{l})$ is one of $\cos ((q)\frac{{2\pi }}{l}),q = 1,2, \ldots ,(l - 1)/2$; otherwise, if $tk > (l - 1)/2$, based on formula (2.21) since $l - tk$  belongs to $[1, (l-1)/2]$.
\end{proof}
   
\section{Fast Algorithm}
\subsection{Fast algorithm for computing fractional period spectrum at period $l/k$}
Based on Theorem 2.5, the algorithm should do a build-in data in the computer. The data contains the coefficients of quadratics, when $k = 1$ and for periods, $l = 2,3, \ldots$, to store in the computer. The built-in data save much computing time because the algorithm does not need to repeatedly find the coefficients for next user.
    
The programs shall design two functions. One function is to obtain the congruence derivative sequence (CDS) (y). The other function $Z(l, q)$ is to yield the autocorrelation of the congruence derivative sequence with two parameters $l$ and $q$.

A MATLAB function $rem(p,l)$ may be used for $p$ and $l$ are integer number if the programming is written by MATLAB, where the $rem(p,l)$ means the remainder after $p$ divided by $l$.

Given input original sequence $x$, the function CDS (y) yields the congruence derivative sequence of the sequence $x$ for a period $l$ in the computation.
      
Suppose we need to calculate the fractional period $l/k$, we then have the following algorithm. 

If $l$ is odd number, complete the loop for ordering to take out $(l-1)/2$ coefficients of quadratics, 
$\cos ((1)\frac{{2\pi }}
{l}),\cos ((2)\frac{{2\pi }}
{l}), \cdots  \cdots ,\cos (((l - 1)/2)\frac{{2\pi }}
{l})$,
from the build-in data. From $t=1$ to $(l-1)/2$, do if $r = rem (tk,l) > (l-1)/2$  then take $\cos ((l - r)\frac{{2\pi }}{l})$ else take $\cos ((r)\frac{{2\pi }}{l})$. At the same time, the $2\cos (()\frac{{2\pi }}{l})$ value times the $Z(l, t)$ and to do summation corresponding to t order until  $t= (l-1)/2$. The computed summation added with $Z(l, 0)$ is the $FPS(l/k)$.

If $l$ is even number, doing the loop for ordering to take out $l/2$ coefficients of quadratics,  
$\cos ((1)\frac{{2\pi }}
{l}),\cos ((2)\frac{{2\pi }}
{l}), \cdots  \cdots ,\cos ((l/2)\frac{{2\pi }}
{l}))$
, from the build-in data, \textit{i.e.}, $t=1$ to $(l-1)/2$ do if $r = rem (tk,l) > (l)/2$ then take $\cos (l - r)\frac{{2\pi }}{l}$ else take $\cos((r)\frac{{2\pi }}{l})$. At the same time, the $2\cos(()\frac{{2\pi }}{l})$ value times the $Z(l, t)$ to do summation corresponding to $t$ order until $t= (l/2-1)$. $FPS (l/k)$ will be the computed summation added with two quantities, $2\cos ((l/2)\frac{{2k\pi }}{l})$ timing $\sum\limits_{t = 0}^{l/2 - 1} {y_t y_{t + l/2} }$, and $Z(l, 0)$. 

To assess the effectiveness of proposed FPS algorithm in capturing fractional periodicities in digital signals, we compute the FPS spectrum of the simulated sinusoidal signal. The sinusoidal signal of length $N=300$, consists of $sine$ and $cosine$ signals with fractional periodicities 3.7 and 5.6, respectively. The signal is corrupted by white Gaussian noise (Equation (3.1)). 
 \begin{equation}
\begin{array}{l}
x(n) = \sin (2\pi \frac{n}{{3.7}} + \frac{\pi }{4}) + \cos (2\pi \frac{n}{{5.6}} + \frac{{3\pi }}{4}) + noise \\ 
n = 1,2, \cdots ,300 \\ 
\end{array}
 \end{equation}
Since the periodic signal is buried by the white Gaussian noise, two periodicities 3.7 and 5.6 in the original signal are hidden by noise (Fig.1 (a)). The proposed FPS spectrum of the signal can clearly discover two pronounced peaks at positions 3.7 and 5.6 (Fig.1 (b)), corresponding to two periodicities 3.7 and 5.6 in the original signal. The power spectrum of this simulation sequence by the proposed method agrees with the DFT (Fig.1 (b) and (c)). This result demonstrates the effectiveness of the FPS algorithm.
\begin{figure}[tbp]
  	\centering
  	\subfloat[]{\includegraphics[width=3.30in]{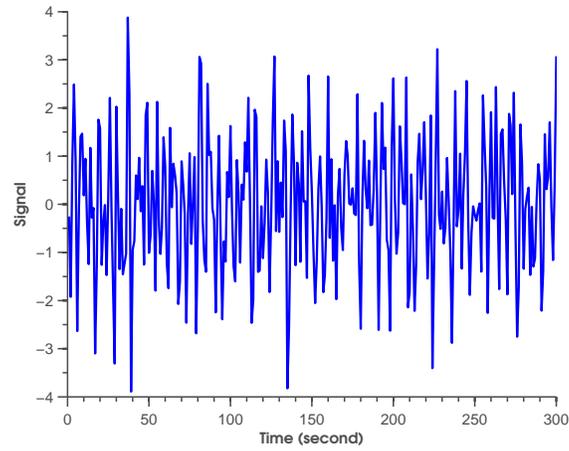}}\quad
  	\subfloat[]{\includegraphics[width=3.30in]{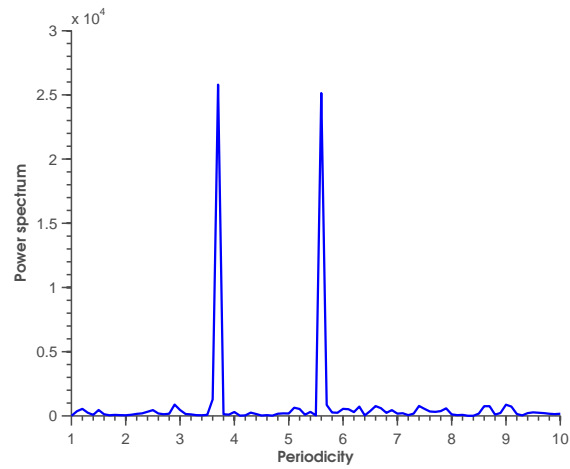}}\quad
    \subfloat[]{\includegraphics[width=3.30in]{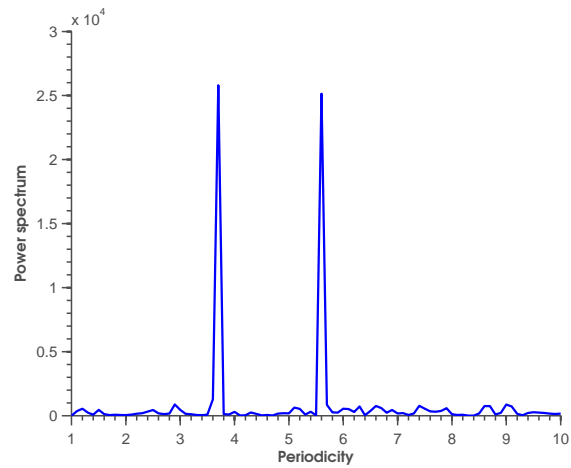}}\quad 
  	\caption{Fractional period spectrum analysis of sinusoidal signal collapsed with white noise. (a) original sinusoidal signal. (b) Fractional period spectrum of the sinusoidal signal by the FPS algorithm. (c) Fractional period spectrum of the sinusoidal signal by DFT.}
  	\label{fig:sub1}
  \end{figure}
    
\subsection{Computational complexity of the FPS algorithm}
Recalling the process for computing $FPS(l/k)$ of a time series by using the new algorithm, the first step is to yield the congruence derivative sequence $y_j$, $j=0,1,.....,l-1$, from the original sequence $x$ of length $m$. This computing process only needs plus operation and do not use multiplication. The second step in the algorithm is to compute $z_{l,q}$ when $l$ is an even or odd number, the algorithm needs approximate $l/2(l+1)$ multiplications. Finally, using formulas (2.18, 2.19) needs $l/2$ multiplications to determine the value of $z_{l,q}cos((q)2k \pi /l)$. So the total computational quantity for the FPS algorithm involves $l/2(l+1) + l/2 = l^2/2 + l$ multiplications, \textit{i.e.}, the computational complexity is then $O(l^2+l)$. 
          
Using the traditional Fourier transform to compute $FPS(l/k)$, we need to calculate the Fourier transform at period $l/k$ by the formulas 2.3 and 2.4. From (2.3), we must calculate m exponents, $ - i2\pi jk/l,j = 0,1,2, \ldots  \ldots ,m - 1$, and call exp function m times to figure out $e^{ - i2\pi jk/l}, j = 0,1,2, \ldots  \ldots ,m - 1$, if using MATLAB. It is known that to compute function $e^{ - i2\pi jk/l}$ costs much time. Then to figure out $\sum\limits_{j = 0}^{m - 1} {e^{ - i2\pi jk/l} } $ needs $2m$ multiplications and the time for calling m exp functions. So, according to formula (2.4) calculating out the Fourier spectrum at period $l/k$ needs $2m +1$ multiplications, \textit{i.e.}, $O(2m+1)$. If $FPS(l/k)$ is computed for whole fractional periods, traditional Fourier transform needs $l/2O(2m+1)$ operations, \textit{i.e.} $O(lm+l)$.

In summary, the analysis of computational complexity indicates that the FPS algorithm is much faster than traditional Fourier transform for computing fractional periodic spectrum because of $O(l^2  + l) <  < O(lm + l)$.
  
\section{Conclusions}
A fast algorithm for computing the FPS of biological sequences is introduced in this paper. First, the mathematical deduction and proof are presented. Then, based on the mathematical theory, the technique of built-in data is suggested and some functions are introduced, so that a simple program can be designed. The analysis of computational complicity shows that this fast algorithm is much faster than traditional Fourier transform.
  
The fast algorithm will be also helpful in computing the FPS for any time series from the science and technology field other than bioinformatics.

In addition, the congruence derivative sequence of original sequences can be used to yield Ramanujan transform easily \citep{hua2014discrete,yin2014novel}, which do not have any computing cost. The computation for identifying periods of the biological sequence or other time series would be done by Fourier transform and Ramanujan transform simultaneously.

\newpage

\bibliographystyle{elsarticle-harv}
\bibliography{../References/myRefs}

\begin{thebibliography}{25}
\expandafter\ifx\csname natexlab\endcsname\relax\def\natexlab#1{#1}\fi
\expandafter\ifx\csname url\endcsname\relax
  \def\url#1{\texttt{#1}}\fi
\expandafter\ifx\csname urlprefix\endcsname\relax\def\urlprefix{URL }\fi

\bibitem[{Anastassiou(2001)}]{anastassiou2001genomic}
Anastassiou, D., 2001. Genomic signal processing. Signal Processing Magazine,
  IEEE 18~(4), 8--20.

\bibitem[{Arqu{\`e}s and Michel(1987)}]{arques1987periodicities}
Arqu{\`e}s, D.~G., Michel, C.~J., 1987. Periodicities in introns. Nucleic Acids
  Research 15~(18), 7581--7592.

\bibitem[{Chen et~al.(2003)Chen, Li, Sun, and Kim}]{chen2003will}
Chen, J., Li, H., Sun, K., Kim, B., 2003. How will bioinformatics impact signal
  processing research? Signal Processing Magazine, IEEE 20~(6), 106--206.

\bibitem[{Eisenberg et~al.(1984)Eisenberg, Weiss, and
  Terwilliger}]{EISENBERG1984}
Eisenberg, D., Weiss, R., Terwilliger, T., 1984. The hydrophobic moment detects
  periodicity in protein hydrophobicity. Proc. Natl. Acad. Sci. 81, 140--144.

\bibitem[{Gruber et~al.(2005)Gruber, S{\"o}ding, and Lupas}]{gruber2005repper}
Gruber, M., S{\"o}ding, J., Lupas, A.~N., 2005. Repper repeats and their
  periodicities in fibrous proteins. Nucleic Acids Research 33~(suppl 2),
  W239--W243.

\bibitem[{Hua et~al.(2014)Hua, Wang, and Zhao}]{hua2014discrete}
Hua, W., Wang, J., Zhao, J., 2014. Discrete ramanujan transform for
  distinguishing the protein coding regions from other regions. Molecular and
  cellular probes 28~(5), 228--236.

\bibitem[{Leonov and Arkin(2005)}]{leonov2005periodicity}
Leonov, H., Arkin, I.~T., 2005. A periodicity analysis of transmembrane
  helices. Bioinformatics 21~(11), 2604--2610.

\bibitem[{Messaoudi et~al.(2013)Messaoudi, Oueslati, and
  Lachiri}]{messaoudi2013detection}
Messaoudi, I., Oueslati, A.~E., Lachiri, Z., 2013. Detection of the 6.5-base
  periodicity in the c. elegans introns based on the frequency chaos game
  signal and the complex morlet wavelet analysis. International Journal of
  Scientific Engineering and Technology 2~(12), 1247--1251.

\bibitem[{Salih and Trifonov(2015)}]{salih2015strong}
Salih, B., Trifonov, E.~N., 2015. Strong nucleosomes reside in meiotic
  centromeres of c. elegans. Journal of Biomolecular Structure and Dynamics
  33~(2), 365--373.

\bibitem[{Silverman and Linsker(1986)}]{silverman1986measure}
Silverman, B., Linsker, R., 1986. A measure of {DNA} periodicity. Journal of
  theoretical biology 118~(3), 295--300.

\bibitem[{Tiwari et~al.(1997)Tiwari, Ramachandran, Bhattacharya, Bhattacharya,
  and Ramaswamy}]{tiwari1997prediction}
Tiwari, S., Ramachandran, S., Bhattacharya, A., Bhattacharya, S., Ramaswamy,
  R., 1997. Prediction of probable genes by {F}ourier analysis of genomic
  sequences. Bioinformatics 13~(3), 263--270.

\bibitem[{Trifonov(1998)}]{trifonov19983}
Trifonov, E.~N., 1998. 3-, 10.5-, 200-and 400-base periodicities in genome
  sequences. Physica A: Statistical Mechanics and its Applications 249~(1),
  511--516.

\bibitem[{Wang et~al.(2014)Wang, Dang, and Yin}]{wang2014SNR}
Wang, J., Dang, C., Yin, C., 2014. A comparative study of the signal-to-noise
  ratios of different representations for symbolic sequences. Journal of
  Applied Mathematics and Bioinformatics 372, 135--145.

\bibitem[{Wang et~al.(2012)Wang, Liu, and Zhao}]{wang2012Some}
Wang, J., Liu, G., Zhao, J., 2012. Some features of {F}ourier spectrum for
  symbolic sequences. Numerical Mathematics A Journal of Chinese
  Universities~(4), 341--356.

\bibitem[{Wang and Yan(2013)}]{wang2013numerical}
Wang, J., Yan, M., 2013. Numerical Methods in Bioinformatics: An Introduction.
  Science Press.

\bibitem[{Welch(1967)}]{welch1967use}
Welch, P., 1967. The use of fast fourier transform for the estimation of power
  spectra: a method based on time averaging over short, modified periodograms.
  IEEE Transactions on Audio and Electroacoustics 15~(2), 70--73.

\bibitem[{Yau et~al.(2003)Yau, Wang, Niknejad, Lu, Jin, and Ho}]{yau2003dna}
Yau, S. S.-T., Wang, J., Niknejad, A., Lu, C., Jin, N., Ho, Y.-K., 2003. {DNA}
  sequence representation without degeneracy. Nucleic Acids Research 31~(12),
  3078--3080.

\bibitem[{Yin(2015)}]{yin2014representation}
Yin, C., 2015. Representation of {DNA} sequences in genetic codon context with
  applications in exon and intron prediction. Journal of Bioinformatics and
  Computational Biology 13~(2), 1550004--2.

\bibitem[{Yin et~al.(2014{\natexlab{a}})Yin, Chen, and Yau}]{yin2014measure}
Yin, C., Chen, Y., Yau, S. S.-T., 2014{\natexlab{a}}. A measure of {DNA}
  sequence similarity by {F}ourier transform with applications on hierarchical
  clustering. Journal of Theoretical Biology 359, 18--28.

\bibitem[{Yin and Wang(2016)}]{yin2016periodic}
Yin, C., Wang, J., 2016. Periodic power spectrum with applications in detection
  of latent periodicities in {DNA} sequences. Journal of Mathematical Biology
  73~(5), 1053--1079.

\bibitem[{Yin and Yau(2005)}]{yin2005fourier}
Yin, C., Yau, S. S.-T., 2005. A {F}ourier characteristic of coding sequences:
  origins and a non-{F}ourier approximation. Journal of Computational Biology
  12~(9), 1153--1165.

\bibitem[{Yin and Yau(2007)}]{yin2007prediction}
Yin, C., Yau, S. S.-T., 2007. Prediction of protein coding regions by the
  3-base periodicity analysis of a {DNA} sequence. Journal of Theoretical
  Biology 247~(4), 687--694.

\bibitem[{Yin and Yau(2015)}]{yin2015improved}
Yin, C., Yau, S. S.-T., 2015. An improved model for whole genome phylogenetic
  analysis by {F}ourier transform. Journal of Theoretical Biology 359~(21),
  18--28.

\bibitem[{Yin and Yau(2017)}]{yin2017coevolution}
Yin, C., Yau, S. S.-T., 2017. A coevolution analysis for identifying
  protein-protein interactions by {F}ourier transform. PloS ONE 12~(4),
  e0174862.

\bibitem[{Yin et~al.(2014{\natexlab{b}})Yin, Yin, and Wang}]{yin2014novel}
Yin, C., Yin, X.~E., Wang, J., 2014{\natexlab{b}}. A novel method for
  comparative analysis of {DNA} sequences by {R}amanujan-{F}ourier transform.
  Journal of Computational Biology 21~(12), 867--879.

\end{thebibliography}


\begin{thebibliography}{}
\bibitem[Bofelli {\it et~al}., 2000]{Boffelli03} Bofelli,F., Name2, Name3 (2003) Article title, {\it Journal Name}, {\bf 199}, 133-154.

\bibitem[Bag {\it et~al}., 2001]{Bag01} Bag,M., Name2, Name3 (2001) Article title, {\it Journal Name}, {\bf 99}, 33-54.

\bibitem[Yoo \textit{et~al}., 2003]{Yoo03}
Yoo,M.S. \textit{et~al}. (2003) Oxidative stress regulated genes
in nigral dopaminergic neurnol cell: correlation with the known
pathology in Parkinson's disease. \textit{Brain Res. Mol. Brain
Res.}, \textbf{110}(Suppl. 1), 76--84.

\bibitem[Lehmann, 1986]{Leh86}
Lehmann,E.L. (1986) Chapter title. \textit{Book Title}. Vol.~1, 2nd edn. Springer-Verlag, New York.

\bibitem[Crenshaw and Jones, 2003]{Cre03}
Crenshaw, B.,III, and Jones, W.B.,Jr (2003) The future of clinical
cancer management: one tumor, one chip. \textit{Bioinformatics},
doi:10.1093/bioinformatics/btn000.

\bibitem[Auhtor \textit{et~al}. (2000)]{Aut00}
Auhtor,A.B. \textit{et~al}. (2000) Chapter title. In Smith, A.C.
(ed.), \textit{Book Title}, 2nd edn. Publisher, Location, Vol. 1, pp.
???--???.

\bibitem[Bardet, 1920]{Bar20}
Bardet, G. (1920) Sur un syndrome d'obesite infantile avec
polydactylie et retinite pigmentaire (contribution a l'etude des
formes cliniques de l'obesite hypophysaire). PhD Thesis, name of
institution, Paris, France.

\end{thebibliography}

\end{document}